\documentclass[10pt]{amsart}
\usepackage{amsmath,amssymb,latexsym,soul}
\usepackage{cite}
\usepackage{color,enumitem,graphicx}
\usepackage[colorlinks=true,urlcolor=blue,
citecolor=red,linkcolor=blue,linktocpage,pdfpagelabels,
bookmarksnumbered,bookmarksopen]{hyperref}
\usepackage[english]{babel}
\usepackage[T1]{fontenc}

\usepackage[left=2.6cm,right=2.6cm,top=2.71cm,bottom=2.71cm]{geometry}
\usepackage[hyperpageref]{backref}
\usepackage[colorinlistoftodos]{todonotes}

\makeatletter
\providecommand\@dotsep{5}
\def\listtodoname{List of Todos}
\def\listoftodos{\@starttoc{tdo}\listtodoname}
\makeatother

\newcommand{\eps}{\varepsilon}

\newcommand{\la}{\lambda}

\newcommand\erre{\mathbb{R}}
\newcommand\ci{\mathbb{C}}
\newcommand{\rn}{\erre^{n}}
\newcommand{\pa}{\partial}

\numberwithin{equation}{section}
\newtheorem{theorem}{Theorem}[section]
\newtheorem{lemma}[theorem]{Lemma}

\newtheorem{definition}[theorem]{Definition}
\newtheorem{proposition}[theorem]{Proposition}

\newcommand{\R}{{\mathbb R}}

\newcommand{\indentitem}{\setlength\itemindent{18pt}}

\title[Logarithmic Schr\"odinger equation]{On the logarithmic Schr\"odinger equation}

\author[P.\ d'Avenia]{Pietro d'Avenia}
\address{Dipartimento di Meccanica, Matematica e Management,
Politecnico di Bari
\newline\indent
Via Orabona 4,  70125  Bari, Italy}
\email{p.davenia@poliba.it}

\author[E.\ Montefusco]{Eugenio Montefusco}
\address{Dipartimento di Matematica,
{\it Sapienza} Universit\`a di Roma 
\newline\indent
Piazzale Aldo Moro 5, 00185 Roma, Italy}
\email{eugenio.montefusco@uniroma1.it}

\author[M.\ Squassina]{Marco Squassina}
\address{Dipartimento di Informatica,
Universit\`a degli Studi di Verona
\newline\indent
C\'a Vignal 2, Strada Le Grazie 15, 37134 Verona, Italy}
\email{marco.squassina@univr.it}

\thanks{The first and third author were supported by MIUR project:
   ``Variational and Topological Methods in the Study of Nonlinear Phenomena''.
    Moreover, the first author was supported also by 2011 FRA 
     project ``Onde solitarie''.  The second author was supported by MIUR project:
   ``Existence, multiplicity and qualitative properties of solutions of 
   nonlinear elliptic problems''}

\begin{document}

\begin{abstract}
In the framework of the nonsmooth critical point theory for lower 
semi-continuous functionals,
we propose a direct variational approach to investigate the 
existence of infinitely many weak solutions for a class
of semi-linear elliptic equations with logarithmic 
nonlinearity arising in physically relevant situations.
Furthermore, we prove that there exists a unique positive solution which is radially symmetric and nondegenerate.
\end{abstract}

\maketitle


\section{Introduction}
\label{sec.introd}
\noindent
The logarithmic Schr\"odinger equation
\begin{equation}
\label{sch-evolv}
i\pa_{t}\phi +\Delta\phi 
	+\phi\log |\phi|^{2}=0, 
\quad
\phi:[0,\infty)\times\rn\to\ci,
\quad
n\geq 3,
\end{equation}
admits applications to quantum mechanics,
quantum optics, nuclear physics, transport and diffusion 
phenomena, open quantum systems, 
effective quantum gravity, theory of 
superfluidity and Bose-Einstein condensation
(see \cite{Zlosh} and the references therein).
We refer to \cite{cazNA,caz1,cazhar} 
for a study 
of existence and uniqueness of the solutions of the associated Cauchy problem 
in a suitable functional framework as well as to a study of the asymptotic behaviour of its solutions and their orbital stability, 
in the spirit of \cite{cl}, with respect to radial perturbations, of the so 
called Gausson solution (see \cite{birula76}).
In this paper we are interested in the existence, multiplicity
and qualitative properties of the standing waves solution of
\eqref{sch-evolv}, i.e. solution in the form
\(\phi=e^{i\omega t} u(x)\), where \(\omega\in\erre\)
and \(u\) is a  real valued function which has to solve
the following semi-linear elliptic problem
\begin{equation}\label{problema}
-\Delta u +\omega u
	= u\log u^{2},  \qquad
u\in H^1(\rn).
\end{equation}
It is well known (see \cite{birula76,birula79}) that the Gausson $$\mathfrak{g}(x)=e^{-|x|^2/2}$$ solves \eqref{problema} for $\omega=-n$.
We emphasize that if $u$ is a solution of
\eqref{problema}, then $\lambda u$, $\lambda\neq 0$, 
is a solution of
$
-\Delta v + \omega' v= v\log v^2
$
with $\omega'=\omega+\log\lambda^2$. 
This fact
allows us to name the solution $\exp\{(\omega+n-|x|^2)/2\}$,
{\em Gausson} for \eqref{problema}.
Moreover, without loss of generality, we can restrict to the case $\omega>0$, even if our results hold for every $\omega\in\mathbb{R}$. 
From a variational point of view, the search of nontrivial solutions 
to \eqref{problema} can be formally associated with the 
study of critical points of the functional on $H^1(\rn)$ defined by
\begin{equation}
\label{functional}
J(u)=\frac{1}{2}\int |\nabla u|^2 
+\frac{\omega+1}{2}\int u^2-\frac{1}{2}\int u^2 \log u^2.
\end{equation} 
Due to the logarithmic Sobolev inequality  
\begin{equation}
\label{eq:logsob}
\int u^2 \log u^2 \leq \frac{a^2}{\pi}  \|\nabla u\|_2^2 
+  (\log \| u\|_2^2 - n (1+\log a))\| u\|_2^2,
\quad 
\hbox{for }u\in H^1(\rn) \hbox{ and } a>0,
\end{equation}
(see e.g. \cite{LLoss}), we have $J(u)>-\infty$ for all $u\in H^1(\rn)$, 
but there are elements $u\in H^1(\rn)$ such that
$ \int u^2 \log u^2=-\infty $.
Thus, in general, $J$ fails to be finite and 
$C^1$ on $H^1(\rn)$. Due to this loss
of smoothness, in order to study existence 
of solutions to \eqref{problema}, to the best of our knowledge, two indirect approaches
were followed so far in the literature. 
On one hand, in \cite{cazNA}, the idea is
to work on the Banach space
\begin{equation}
\label{Wcaz}
W=\big\{u\in H^1(\rn)\;\vert \int u^2 |\log u^2|<\infty\big\},
\,\,\,
\|u\|_W=\|u\|_{H^1}+\inf\Big\{\gamma>0\;\vert\int A(\gamma^{-1}|u|)\leq 1\Big\},
\end{equation}
where $A(s)=-s^2\log s^2$ on $[0,e^{-3}]$ and $A(s)=3s^2+4 e^{-3}s-e^{-6}$
on $[e^{-3},\infty)$.
In fact, it turns out that, in this framework
$J:W\to\R$ is well defined and $C^1$ smooth (see \cite[Proposition 2.7]{cazNA}).
On the other hand,
in \cite{guerrero}, the authors
penalize the nonlinearity around the origin and try to obtain
a priori estimates to get a nontrivial solution 
at the limit. 
However, the drawback of these indirect approaches, is that
the Palais-Smale condition cannot be obtained, due to a loss 
of coercivity of the functional $J$,
and, in general, no multiplicity result can be obtained 
by the Lusternik-Schnirelmann category theory.
In this paper we introduce a direct approach to study 
the existence of infinitely many 
weak solutions to~\eqref{problema}, in the framework 
of the nonsmooth critical point theory 
developed by Degiovanni-Zani in \cite{DegZan96,DegZan2000} 
(see also~\cite{CamDeg}) for suitable 
classes of lower semi-continuous functionals, and based on 
the notion of weak slope (see \cite{DegMar,CDM}).
In fact, it is easy to see that the functional 
$J:H^1_{{\rm rad}}(\rn)\to\erre\cup\{+\infty\}$ is 
lower semicontinuous (see Proposition~\ref{semicont}) 
and that it satisfies the Palais-Smale condition
in the sense of weak slope (see Proposition~\ref{comp-cps}). 
More precisely, we shall prove the following 
\begin{theorem}
\label{main}
Problem~\eqref{problema} has a sequence of solutions 
$u_k\in H^1_{{\rm rad}}(\rn)$  with $J(u_k)\to+\infty$ as $k\to+\infty$. 
\end{theorem}
\noindent To the best of our knowledge this is the first 
multiplicity result for \eqref{problema} on the
space $H^1(\rn)$ and it is obtained directly 
{\em without penalizing} the functional and 
{\em without changing the topology} of the space. 
It should also be noticed that, due to the behaviour around zero, 
our logarithmic nonlinearity does not fit into the 
framework of the classical papers by Berestycki and 
Lions \cite{BL1,BL2}. 
We also point out that, even without working in the 
restricted space of radial functions $H^1_{{\rm rad}}(\rn)$,
since $J$ decreases 
under polarization of nonnegative functions
of $H^1(\rn)$, we can obtain 
the existence of a Palais-Smale sequence $\{u_k\}\subset H^1(\rn)$
with the additional information that $\|u_k-|u_k|^*\|_{L^{2^*}(\rn)}\to 0$ 
as $h\to\infty$, namely $\{u_k\}$ is {\em almost}
radially symmetric and decreasing (see \cite[Theorem 3.10]{CCM}).\\
In the last section 
we study some qualitative properties of
the solutions of \eqref{problema}.
We are able to prove that the nonnegative solutions are strictly positive and that
they are smooth. 
By exploiting the moving plane method (we outline that our nonlinearity is not $C^1$ in $[0,\infty)$), we show 
\begin{theorem}
\label{thm:mpm}
Up to translations, the Gausson for \eqref{problema} is the unique strictly positive $C^2$-solution  such that $u(x)\to 0$ as $|x|\to \infty$.  
\end{theorem}
\noindent Then we get that the first solution $u_1$ in Theorem \ref{main} is the Gausson for \eqref{problema}. Moreover
we prove
\begin{theorem}
\label{main2}
The Gausson $\mathfrak{g}$ is nondegenerate, that is 
$\operatorname{Ker}(L)=\operatorname{span}
\big\{\partial_{x_h}\mathfrak{g}\big\}$,
where $Lu=-\Delta u + ( |x|^2 -n-2) u$ is the linearized 
operator for $-\Delta u  - n u = u \log u^2$ at $\mathfrak{g}$.
\end{theorem}
\noindent Finally, in Theorem~\ref{biject}, we also obtain a variational characterization of ground state solutions
(namely minima of $J$ on the Nehari manifold) of the 
problem as minima on the $L^2$-sphere. 
We believe that the nondegeneracy of $\mathfrak{g}$ and the  connection between the minimization on the Nehari manifold and on the $L^2$-sphere  can be 
useful in the study of the stability properties
of the logarithmic Schr\"odinger equation~\eqref{sch-evolv}, 
possibly in presence of an external driving potential (see e.g. \cite{BJ}).

\subsection*{Notations} 
\begin{enumerate}
\item $L^\infty_c(\rn)$ denotes the space of functions in $L^\infty(\rn)$ with compact support;
\item $H^1_{\rm rad}(\rn)$ denotes the space of $H^1(\rn)$ functions that are radially symmetric;
\item $C$ denotes a generic positive constant which can changes from line to line.
\end{enumerate}

\section{The multiplicity result}
\noindent
The aim of this section is to prove Theorem \ref{main}.

\subsection{Recalls of nonsmooth critical point theory}
Let us recall some notions useful in the following. 
For a more complete treatment of these arguments 
we refer the reader to \cite{CamDeg,DegMar,DegZan2000}.
Let $(X,\|\cdot\|_X)$ be a Banach space and 
$f:X\to\bar{\mathbb{R}}$ be a function. We consider
\[
\operatorname{epi}(f)=\left\{ (x,\lambda)
\in X\times\mathbb{R}\; \vert \; f(x)\leq\lambda\right\}
\]
endowed with the metric induced by the norm 
$\|\cdot\|_{X\times\mathbb{R}}=(\|\cdot\|_X^2 + |\cdot|^2)^{1/2}$ of $X\times{\mathbb R}$
and we denote with $B_\delta(x,\lambda)$ the open ball of 
center $(x,\lambda)$ and radius $\delta>0$.
Moreover we give the following definitions.
First we give the notion of weak slope for continuous functions.
\begin{definition}
Let $f:X\to\mathbb{R}$ be continuous. For every $x\in X$, we denote 
$|df|(x)$ the supremum of the $\sigma$'s 
in  $[0,+\infty[$ such that there exist $\delta>0$ and a continuous map
$\mathcal{H}:B_\delta(x)\times [0,\delta] \to X$,
satisfying
\[
\| \mathcal{H}(w,t) - w\|_{X} \leq t,
\qquad
f(\mathcal{H}(w,t))\leq f(w) - \sigma t,
\]
whenever $w\in B_\delta(x)$ and $t\in [0,\delta]$.
The extended real number $|df|(x)$ is 
called the {\em weak slope} of $f$ at $x$.
\end{definition}
\noindent Now, let us consider the function $\mathcal{G}_f:= (x,\lambda)\in \operatorname{epi}(f) \mapsto \lambda\in\mathbb{R}$. The function  $\mathcal{G}_f$ is continuous and Lipschitzian of 
constant $1$ and it allows to generalize the notion of weak slope for non-continuous functions $f$ as follows.
\begin{definition}
For all $x\in X$ with $f(x)\in\mathbb{R}$ 
\begin{equation*}
|df|(x):=
\begin{cases}
\frac{|d\mathcal{G}_f|(x,f(x))}
{\sqrt{1-|d\mathcal{G}_f|(x,f(x))^2}}
&
\hbox{if } |d\mathcal{G}_f|(x,f(x))<1,\\
+\infty &
\hbox{if } |d\mathcal{G}_f|(x,f(x))=1. 
\end{cases}
\end{equation*}
\end{definition}
\noindent We also need the following
\begin{definition}
Let $c\in\mathbb{R}$. The function $f$ satisfies (epi)$_c$ condition if there exists $\varepsilon>0$ such that
\[
\inf\{|d\mathcal{G}_f|(x,\lambda)\;\vert\; f(x)<\lambda,|\lambda-c|<\varepsilon\}>0.
\]
\end{definition}
\begin{definition}
$x\in X$ is a (lower) critical point of $f$ 
if $f(x)\in\mathbb{R}$ and $|df|(x)=0$.
\end{definition}
\begin{definition}
Let $c\in\mathbb{R}$. A sequence $\{x_k\}\subset X$ 
is a Palais-Smale 
sequence for $f$ at level $c$ if $f(x_k)\to c$ and
$|df|(x_k) \to 0$.
Moreover $f$ satisfies the Palais-Smale 
condition at level $c$ if every Palais-Smale 
sequence for $f$ at level $c$ admits a 
convergent subsequence in $X$. 
\end{definition}
\begin{definition}
Let $f$ be even with $f(0)\in\mathbb{R}$. For every $\lambda\geq f(0)$, we denote 
$|d_{\mathbb{Z}_2}\mathcal{G}_f|(0,\lambda)$ the supremum of the $\sigma$'s 
in  $[0,+\infty[$ such that there exist $\delta>0$ and a continuous map
$\mathcal{H}=(\mathcal{H}_1,\mathcal{H}_2):(B_\delta(0,\lambda)\cap\operatorname{epi}(f))\times [0,\delta] \to \operatorname{epi}(f)$,
satisfying
\[
\| \mathcal{H}((w,\mu),t) - (w,\mu)\|_{X\times\mathbb{R}} \leq t,
\qquad
\mathcal{H}_2((w,\mu),t)\leq \mu - \sigma t,
\qquad
\mathcal{H}_1((-w,\mu),t)=-\mathcal{H}_1((w,\mu),t),
\]
whenever $(w,\mu)\in B_\delta(0,\lambda)\cap\operatorname{epi}(f)$ 
and $t\in [0,\delta]$.
\end{definition}
\noindent We will apply the following abstract result (see \cite{DegZan2000}).
\begin{theorem}
\label{astratto}
Let $X$ be a Banach space 
and $f:X\to\bar{\mathbb{R}}$ a lower semicontinuous even functional.
Assume that $f(0)=0$ and there exists a  strictly increasing sequence $\{V_k\}$ of finite-dimensional subspaces of $X$ with the following properties:
\begin{enumerate}[label={\em(GH\arabic*)},ref=GH\arabic*]
\item \label{it:G1} there exist a closed subspace $Z$ of $X$, $\rho >0$ and $\alpha>0$ such that $X=V_0\oplus Z$ and for every $x\in Z$ with $\|x\|_X =\rho$, $f(x)\geq\alpha$;
\item \label{it:G2} there exists a sequence $\{R_k\}\subset]\rho,+\infty[$ such that for any $x\in V_k$ with $\|x\|_X \geq R_k$, $f(x)\leq 0$.
\end{enumerate}
Moreover, assume that
\begin{enumerate}[label={\em(PSH)},ref=PSH]
\item \label{it:PS} for every $c\geq\alpha$, the function $f$ satisfies the Palais-Smale condition at level $c$ and (epi)$_c$ condition;
\end{enumerate} 
\begin{enumerate}[label={\em(WSH)},ref=WSH]
\item \label{it:WSH} $|d_{\mathbb{Z}_2}\mathcal{G}_f|(0,\lambda)\neq 0$, whenever $\lambda\geq\alpha$.
\end{enumerate} 
Then there exists a sequence $\{x_k\}$ of critical points of $f$ such that $f(x_k)\to+\infty$.
\end{theorem}

\subsection{The Palais-Smale condition}
In this subsection we prove the properties of the functional $J$ that will be useful in the last part of section.
First we establish the relation between the weak slope of the functional $J$ and its directional derivatives (along {\em admissible} directions). 
In the following, we shall denote by $g$ and $G$ 
the extensions by continuity of the functions $s \log s^2$ and  $s^2 \log s^2$ respectively and
$G_1$ and $G_2$ the continuous functions
$$
G_1(s):=(s^2\log s^2)^-
\quad
\text{and}
\quad 
G_2(s):=(s^2\log s^2 )^+.
$$
Observe that, if $u\in H^1_{\rm loc}(\rn)$, then for every $v\in H^1(\rn) \cap L^\infty_c(\rn)$, $g(u)v\in L^1(\rn)$, since
\[
\int |g(u)v| 
\leq  C \Big(1 + \int_{{\operatorname{supp} v}\cap\{|u|> 1\}} |u|^{1+\delta}\Big) < +\infty,
\quad\text{for some $\delta\in(0, 2^*-1]$,}
\]
and so, in particular, $g(u)\in L^1_{{\rm loc}}(\rn)$.
\noindent
If $u\in H^1(\rn)$ and $v\in H^1(\rn) \cap L^\infty_c(\rn)$, we can consider
\begin{equation}
\label{equat-var}
\langle J'(u),v\rangle:= \int \nabla u\cdot\nabla v 
+ \omega \int uv -  \int uv\log u^2.
\end{equation}


\noindent
We have the following 
\begin{proposition}
\label{linkder-pend}
Let $u\in H^1(\rn)$ with $J(u)\in\mathbb{R}$ and $|dJ|(u)<+\infty$. 
Then the following facts hold:
\begin{enumerate}
\item \label{duality} $g(u)\in L^1_{{\rm loc}}(\rn)\cap H^{-1}(\rn)$ and for any $v\in H^1(\rn) \cap L^\infty_c(\rn)$, we have \begin{equation}
\label{eq:connection}
|\langle J'(u),v\rangle|\leq |dJ|(u) \|v\|;
\end{equation}
\item \label{tests} if $v\in H^1(\rn)$ is such that $(g(u)v)^+\in L^1(\rn)$ or $(g(u)v)^-\in L^1(\rn)$, 
then $g(u)v\in L^1(\rn)$ and identity \eqref{equat-var} holds, identifying $J'(u)$ as an element in $H^{-1}(\rn)$.
\end{enumerate}
\end{proposition}
\begin{proof}
Recalling the notion of subdifferential in \cite{CamDeg} and, by \cite[Theorem 4.13]{CamDeg}, we have $\partial J(u)\not=\emptyset$
and $|dJ|(u)\geq \min\{\|\alpha\|_*\;\vert\;\alpha \in \partial J(u)\}$ where $\|\cdot\|_*$ is the norm in $H^{-1}(\rn)$.
Now let
\[
T(u)=\dfrac{1}{2}\|\nabla u\|_2^2 + \frac{\omega + 1}{2} \| u\|_2^2 
\quad
\hbox{ and }
\quad
Q(u)= -\frac{1}{2}\int u^2 \log u^2.
\]
By \cite[Corollary 5.3]{CamDeg}
we have $\partial J(u) \subset \partial T(u) +\partial Q(u)$
and, since $\partial J(u)\not=\emptyset$, then $\partial Q(u)$ is nonempty too. Hence, in light of \cite[(b) of Theorem 3.1]{DegZan2000}, we get that $-u-g(u)\in L^1_{{\rm loc}}(\rn)\cap H^{-1}(\rn)$, and then $g(u)\in L^1_{{\rm loc}}(\rn)\cap H^{-1}(\rn)$,
and $\partial Q(u)=\{-u \log u^2- u\}$.
Thus, taking into account that $\partial J(u)=\{J'(u)\}$, with $J'(u)$ as in \eqref{equat-var}, we get
\eqref{eq:connection}.
Assertion \eqref{tests} follows by the result of \cite{BrezBrow}.
\end{proof}

\begin{proposition}
\label{semicont}
The functional $J$ is lower semicontinuous.

\end{proposition}
\begin{proof}
Assume that $\{u_k\}\subset H^1(\rn)$ is a sequence converging to some $u$. Up to a subsequence,
$G_1(u_k)$ converges pointwise to $G_1(u)$. Hence, by virtue of Fatou's Lemma, we get
$$
\int G_1(u)\leq\liminf_k \int G_1(u_k).
$$
On the other hand, taking into account
that, for any $\delta\in(0,2^+-2]$, 
there exists 
$C_\delta>0$ such that $G_2(s)\leq C_\delta |s|^{2+\delta}$ for all $s\in\R$
and that $u_k\to u$ strongly in $L^{2+\delta}(\R^N)$, we conclude that
$$
\int G_2(u)=\lim_k \int G_2(u_k).
$$
Hence, as $G(s)=G_2(s)-G_1(s)$, the desired conclusion follows.
\end{proof}

\begin{proposition}
\label{comp-cps}
$J|_{H^1_{{\rm rad}}(\rn)}$ satisfies the Palais-Smale condition at level $c$ for every $c\in\mathbb{R}$.
\end{proposition}
\begin{proof}
Let us first prove that the Palais-Smale sequences of $J$ are bounded in $H^1(\rn)$.
Let $\{u_k\}\subset H^1(\rn)$ be a Palais-Smale sequence of $J$, namely
$J(u_k)\to c$ and $|dJ|(u_k) \to 0$.
By Proposition~\ref{linkder-pend}, we have that $\langle J'(u_k),v\rangle=o(1)\|v\|$
for any $v\in H^1(\rn) \cap L^\infty_c(\rn)$,
namely $J'(u_k)\to 0$ in $H^{-1}(\rn)$ as $k\to\infty$.
Now, notice that, if $u\in H^1(\rn)$, then $(u^2\log u^2)^+\in L^1(\rn)$.
Thus, by virtue of \eqref{tests} of Proposition~\ref{linkder-pend}, we are allowed to choose $u_k$ as  admissible test functions
in equation~\eqref{equat-var} and
\begin{equation}
\label{eq:boundL2}
\| u_k\|_2^2 = 2 J(u_k) - \langle J'(u_k),u_k\rangle
\leq 2c+o(1) \|u_k\|.
\end{equation}
By \eqref{eq:logsob} for $a>0$ small, \eqref{eq:boundL2} and the boundedness of $\{J(u_k)\}$, for $\delta>0$ small, we have that
\[
\|u_k\|^2 \leq C+ C(1+ o(1)\|u_k\|)^{1+\delta} + o(1)\|u_k\|,
\]
and so $\{u_k\}$ is bounded in $H^1(\rn)$.
\noindent
Let $\{u_k\}$ now be a Palais-Smale sequence for $J$ in $H^1_{\rm rad}(\rn)$. 
The above argument shows that $\{u_k\}$ is bounded in $H^1_{\rm rad}(\rn)$.
Then, up to a subsequence, there is $u\in H_{\rm rad}^1(\rn)$ with
\[
u_k \rightharpoonup u  \hbox{ in } H^1(\rn),   
\qquad
u_k \to u  \hbox{ in } L^p(\rn),\;2<p<2^*,
\qquad
u_k \to u  \hbox{ a.e. in } \rn.  
\]
We want to prove that
\begin{equation}
\label{solprobl}
\int\nabla u \cdot\nabla v
+\omega\int u v
=\int u v \log u^2,\quad \hbox{for all } v\in H^1(\rn) \cap L^\infty_c(\rn).
\end{equation}
So, fixed $v\in H^1(\rn) \cap L^\infty_c(\rn)$, let us consider 
$\vartheta_R(u_k)v$,
where, given $R>0$, $\vartheta_R:\mathbb{R}\to[0,1]$ is smooth, $\vartheta_R(s)=1$ for $|s|\leq R$,  $\vartheta_R(s)=0$ for $|s|\geq 2R$ and $|\vartheta'_R(s)|\leq C/R$ in $\mathbb{R}$.
Obviously we have that $\vartheta_R(u_k)v\in H^1(\rn) \cap L^\infty_c(\rn)$.
Thus, by~\eqref{equat-var} and   
taking into account the boundeness of $\{u_k\}$, we have
\[
\left| \int \vartheta_R(u_k) \nabla u_k \nabla v
+\omega \int \vartheta_R(u_k) u_k v
-\int \vartheta_R(u_k) u_k v \log u_k^2 
- \langle J'(u_k),\vartheta_R(u_k)v\rangle\right|
\leq \frac{C}{R}.
\]
Passing to the limit as $k\to+\infty$, since
$\vartheta_R(u_k)\nabla v \to \vartheta_R(u)\nabla v$ in $L^2(\rn,\rn)$, $ \vartheta_R(u_k) u_k \log u_k^2 \to \vartheta_R(u) u \log u^2$
a.e. in $\rn$ and taking into account
that $\{\vartheta_R(u_k) u_k \log u_k^2 \}$ 
is bounded in $L^2_{{\rm loc}}(\rn)$, 
we have
\[
\left| \int \vartheta_R(u) \nabla u \nabla v
+\omega \int \vartheta_R(u) u v
-\int \vartheta_R(u) u v \log u^2 \right|
\leq \frac{C}{R}.
\]
Thus we pass to the limit as $R\to +\infty$ and we get~\eqref{solprobl}.
Moreover, as in the proof of Proposition \ref{semicont}, we have that
$$
\limsup_k \int u_k^2\log u_k^2\leq \int u^2\log u^2.
$$
Hence, since $\langle J'(u_k),u_k \rangle \to 0$ and choosing, by 
\eqref{tests} of Proposition~\ref{linkder-pend},  $v=u$ 
in~\eqref{solprobl}, we get
\[
\limsup_k (\|\nabla u_k\|_2^2 +\omega \|u_k\|_2^2)
=\limsup_k \int u_k^2 \log u_k^2\leq \int u^2\log u^2
= \|\nabla u\|_2^2 +\omega \|u\|_2^2,
\]
which implies the convergence of $u_k\to u$ in $H^1_{\rm rad}(\rn)$.
\end{proof}

\subsection{Proof of Theorem~\ref{main}}
\noindent
To prove the existence of sequence $\{u_k\}\subset H^1_{\rm rad}(\rn)$
of (weak) solutions to \eqref{problema} with $J(u_k)\to+\infty$,
we will apply Theorem~\ref{astratto} with $X=H^1_{\rm rad}(\rn)$.
In light of Proposition~\ref{comp-cps}, $J$ satisfies 
the Palais-Smale condition. Moreover J satisfies (epi)$_c$ and (\ref{it:WSH}) conditions (see \cite[Theorem 3.4]{DegZan2000}). Hence, 
it remains to check that $J$ satisfies also the 
geometrical assumptions. Obviously, $J(0)=0$. 
Moreover by the logarithmic Sobolev inequality~\eqref{eq:logsob}, we have that
\[
J(u) \geq  
\frac{1}{2}\left(1 - \frac{a^2}{\pi}\right) \|\nabla u\|_2^2
+ \frac{1}{2} (\omega + 1 + n (1+\log a) - \log \| u\|_2^2)
\| u\|_2^2\\
\geq  c \| u\|^2,
\]
for a suitable $a$ and  if $\|u\|$ are sufficiently small.
Then, if we take $Z=X=H^1_{\rm rad}(\rn)$ and $V_0=\{0\}$ we have (\ref{it:G1}).
Finally, let us consider a strictly increasing 
sequence $\{V_k\}$ of finite-dimensional 
subspaces of $H^1_{{\rm rad}}(\rn)$ constituted 
by bounded functions (for instance, one
can consider the eigenvectors of $-\Delta+|x|^2$, 
see \cite[Chapter 3]{BS}).
Since any norm is equivalent on any $V_k$, if $\{u_m\}\subset V_k$ is such 
that $\|u_m\|\to +\infty$, then also $\mu_m=\|u_m\|_2 \to +\infty$. Write
now $u_m=\mu_m w_m$, where $w_m=\|u_m\|_2^{-1}u_m$. 
Thus $\|w_m\|_2=1$, $\|\nabla w_m\|_2\leq C$
and $\|w_m\|_\infty\leq C$, yielding in turn
\[
J(u_m)
=\frac{\mu_m^2}{2}  \Big(  \|\nabla w_m\|_2^2 
+ \omega + 1 - \log \mu_m^2
- \int w_m^2 \log w_m^2 \Big)  
\leq  \frac{\mu_m^2}{2} (C - \log \mu_m^2)\to-\infty.
\]
Thus, there exist $\{R_k\}\subset]\rho,+\infty[$ 
such that for $u\in V_k$ with $\|u\| \geq R_k$, $J(u)\leq 0$.
Hence, also~\eqref{it:G2} is satisfied and the assertion follows
as, by Proposition~\ref{linkder-pend}, the critical points of 
$J$ in the sense of weak slope correspond to solutions to~\eqref{problema}.

\section{Qualitative properties of the nonnegative solutions}
\label{solprop}

\subsection{Positivity and regularity of solutions}
If we take $\beta(s)=\omega s-s\log s^2$, since $\beta$ is continuous, 
nondecreasing for $s$ small, $\beta(0)=0$ and $\beta(\sqrt{e^\omega})=0$, 
by \cite[Theorem 1]{Vazquez} we have that each solution $u\geq 0$ 
of \eqref{problema} such that $u\in L^1_{\rm loc}(\rn)$ and $\Delta u \in L^1_{\rm loc}(\rn)$ 
in the sense of distribution, is either trivial or stricly positive.
Moreover, observe that any given nonnegative solution to 
equation~\eqref{problema} satisfies the inequality
$$
-\Delta u+\omega u\leq (u \log u^2)^+.
$$
In particular, for any $\delta>0$, there exists $C_\delta>0$ such that 
$$
-\Delta u\leq \ell(u),\qquad \ell(s)=-\omega s+C_\delta s^{1+\delta}
$$
Since we have $|\ell(s)|\leq C(1+|s|^{1+\delta})$ for 
all $s\in{\mathbb R}$ and some $C>0$,
by repeating the argument of the proof of \cite[Lemma B.3]{struwe}, 
it is possible to prove that $u\in L^q_{{\rm loc}}({\mathbb R}^n)$
for every $q<\infty$. Then, by standard regularity arguments, 
the $C^2$ smoothness of $u$ readily follows.

\subsection{Uniqueness of positive solutions}

\noindent
In this subsection we prove Theorem \ref{thm:mpm}.  prove the following
\begin{proof}
First of all, by means of the moving plane method \cite{GNN}, we prove that each positive and vanishing 
classical solution to \eqref{problema} has to be radially symmetric about some point.
Let $u\in C^2({\mathbb R}^n)$ be a solution  to 
equation \eqref{problema} with $u>0$ and  $u(x)\to 0$ as $|x|\to \infty$, $\lambda\in\mathbb{R}$, $\Sigma_\lambda:=\{x\in\rn\vert\,\, x_1<\lambda\}$,
$x_\lambda:=(2\lambda-x_1,x_2,\ldots,x_n)$, $u_\lambda(x):=u(x_\lambda)$ 
and $w_\lambda:=u_\lambda-u$. Then, it is easy to verify that
\[
-\Delta w_\lambda + c_\lambda(x) w_\lambda =0,
\]
where we have set
\[
c_\lambda(x):=-\int_0^1 \left(2-\omega+\log(\sigma u_\lambda(x) 
+ (1-\sigma) u(x))^2\right) d\sigma.
\]
Notice that $x\mapsto c_\lambda(x)$ is possibly unbounded 
from above, but it is bounded from below.
Since $u$ goes to zero at infinity we notice that there 
exists $R>0$ such that $u(x)< \sqrt{e^{\omega-2}}$ in $B^c_R(0)$. 
We claim that for every $\lambda\in\mathbb{R}$ we have 
that $w_\lambda \geq 0$ in $B^c_R(0)$.
Indeed, assume by contradiction that there exist $\lambda\in\mathbb{R}$ 
and points in ${B}^c_R(0)$ at which $w_\lambda <0$. 
Let $\bar x\in B^c_R(0)$ a negative minimum point of 
$w_\lambda$.  Then, we have
\[
0<-(2-\omega+\log u^2 (\bar x))  \leq c_\lambda(\bar x) \leq 
-(2-\omega+\log u_\lambda^2 (\bar x)),
\]
and so $-\Delta w_\lambda (\bar x) \geq 0$ that is a contradiction.
thus, if $\lambda <-R$, we have that $\Sigma_\lambda\subset B^c_R(0)$ 
and then $w_\lambda(x)\geq 0$ for every $x\in\Sigma_\lambda$. 
Now we want to move the hyperplane $\partial \Sigma_\lambda$ to the 
right (i.e.\ increasing the value of $\lambda$) preserving the 
inequality $w_\lambda\geq 0$ to the limit position. 
Let $\lambda_0 := \sup\{\lambda < 0 \,\,\vert\,\, w_\lambda\geq 0 
\hbox{ in } \Sigma_\lambda \}$. 
First of all we observe that, by continuity,
$w_{\lambda_0} \geq 0$ in $\Sigma_{\lambda_0}$. 
Then by the maximum principle (see \cite[Theorem 7.3.3]{chenli}, 
where one can assume that the function $c(x)$ is merely bounded from below), we have that either 
$w_{\lambda_0}\equiv 0$ in $\Sigma_{\lambda_0}$ or $w_{\lambda_0}>0$ 
in the interior of $\Sigma_{\lambda_0}$. We claim that if $\lambda_0 <0$ 
then $w_{\lambda_0}\equiv 0$. We show that $w_{\lambda_0}>0$ in the 
interior of $\Sigma_{\lambda_0}$ implies that
\begin{equation}
\label{eq:mov}
\exists \delta_0>0 \hbox{ such that } \forall\delta\in(0,\delta_0): 
w_{\lambda_0+\delta}\geq 0 \hbox{ in }\Sigma_{\lambda_0+\delta},
\end{equation} 
violating the definition of $\lambda_0$. Indeed, assume by contradiction 
that \eqref{eq:mov} is not true. Then we can consider a 
sequence $\delta_k \to 0$ such that for every $k$ there exists 
a negative minimum point $\bar{x}_k$ of $w_{\lambda_0+\delta_k}$ 
in $\Sigma_{\lambda_0+\delta_k}$.
Then $\bar{x}_k \in \bar{B}_R(0)\cap\Sigma_{\lambda_0+\delta_k}$ 
and $\nabla w_{\lambda_0+\delta_k} (\bar{x}_k)=0$. The boundedness 
of the sequence $\{\bar{x}_k\}$ implies that, up to a subsequence,
$\bar{x}_k\to\bar{x}$ and 
\begin{equation}
\label{eq:mov2}
w_{\lambda_0} (\bar{x})=\lim_k w_{\lambda_0+\delta_k} (\bar{x}_k) \leq 0,
\qquad
\bar{x}\in \overline{\Sigma}_{\lambda_0}
\end{equation}
and
\begin{equation}
\label{eq:mov3}
\nabla w_{\lambda_0} (\bar{x})=\lim_k \nabla w_{\lambda_0+\delta_k} (\bar{x}_k)=0.
\end{equation}
Then, by \eqref{eq:mov2} we have that $\bar{x}\in \partial\Sigma_{\lambda_0}$ 
and $w_{\lambda_0} (\bar{x})=0$. Therefore, by Hopf  Lemma (see again \cite[Theorem 7.3.3]{chenli}) we have that
\[
\frac{\partial w_{\lambda_0}}{\partial {\bf n}}(\bar{x})<0
\]
that contradicts \eqref{eq:mov3}.
If $\lambda_0 = 0$, then we can carry out the above procedure in 
the opposite direction, namely, we move the hyperplane $x_1=\lambda$ 
with $\lambda>0$ in the negative  direction. 
If the infimum of values of such $\lambda$'s is strictly positive, we 
get the symmetry and monotonicity as in the case $\lambda_0<0$. 
If such infimum is $0$ we get obviously the symmetry and monotonicity 
with respect to the hyperplane $x_1=0$. By the arbitrariness of  the choice of the $x_1$ direction, we can  conclude that the solution \(u\) must 
be radially symmetric about some point.
Finally, \cite{serrintang} prove that there exists at most one non-negative non-trivial $C^1$ distribution solution of \eqref{problema} in the class of radial functions which tends to zero at infinity. 
Then, up to translations,  such a solution is $\mathfrak{g}$.
\end{proof}

\subsection{Gausson's nondegeneracy} 
We have shown that $\mathfrak{g}$ is the unique radial positive  solution of the equation
\begin{equation}
\label{eq:gausson}
-\Delta u  - n u = u \log u^2.
\end{equation}
In this subsection we prove Theorem \ref{main2}.
The linearized operator $L$ for \eqref{eq:gausson} at $\mathfrak{g}$ is found to be
\[
Lu=-\Delta u + ( |x|^2 -n-2) u,
\]
acting on $L^2(\rn)$ with domain $H^2(\rn)$.
To prove that $\operatorname{Ker}(L)=\operatorname{span}\big\{\partial_{x_i}\mathfrak{g}\big\}$,
we introduce the following notations. We set
\[
r=|x|, \qquad \vartheta=\frac{x}{|x|}\in \mathbb{S}^{n-1},
\]
and we denote by $\Delta_r$ the Laplace operator in radial coordinates and with $\Delta_{\mathbb{S}^{n-1}}$ the Laplace-Beltrami operator.
Let us consider the spherical harmonics $Y_{k,h}(\vartheta)$, satisfying
\begin{equation}
\label{eq:hs}
-\Delta_{\mathbb{S}^{n-1}} Y_{k,h}=\lambda_k Y_{k,h}.
\end{equation}
We recall that \eqref{eq:hs} admits a sequence of eigenvalues 
$\lambda_k= k(k+n-2)$, $k\in\mathbb{N}$ whose multiplicity is given by $\mu_k - \mu_{k-2}$ 
where 
\[
\mu_k:=
\begin{cases}
\frac{(n+k-1)!}{(n-1)!k!} & \hbox{ for }k\geq 0,\\
0 & \hbox{ for }k<0,
\end{cases}
\]
(see e.g. \cite{BS}). In particular $\lambda_0=0$ and $\lambda_1=n-1$  have, respectively, multiplicity $1$ and  $n$.
For every $u\in H^1(\rn)$, we have that
\begin{equation}
\label{eq:scomp}
u(x)=\sum_{k\in\mathbb{N}} \sum_{h=1}^{\mu_k-\mu_{k-2}} \psi_{k,h}(r) Y_{k,h}(\vartheta),
\quad\hbox{ where }
\psi_{k,h}(r):=\int_{\mathbb{S}^{n-1}} u(r\vartheta) Y_{k,h}(\vartheta) d\vartheta,
\end{equation}
and, for every $k\in\mathbb{N}$ and $h\in\{1,\ldots,\mu_k-\mu_{k-2}\}$,
\begin{equation}
\label{eq:deltak}
\Delta(\psi_{k,h} Y_{k,h})= Y_{k,h}(\vartheta) \Delta_r \psi_{k,h} (r)+ \frac{1}{r^2} \psi_{k,h}(r) \Delta_{\mathbb{S}^{n-1}} Y_{k,h}(\vartheta).
\end{equation}
Therefore, by combining formulas \eqref{eq:hs}, \eqref{eq:scomp} and \eqref{eq:deltak}, we have that $u\in \operatorname{Ker}(L)$ 
if and only if, for every $k\in\mathbb{N}$ and all $h\in\{1,\ldots,\mu_k-\mu_{k-2}\}$,
\begin{equation}
\label{eq:ndeg}
A_k(\psi_{k,h})=0,
\end{equation}
where
\[
A_k(\psi)=-\psi'' - \frac{n-1}{r} \psi' + \Big( r^2+ \frac{\lambda_k}{r^2} - n -2\Big) \psi .
\]
For the spectral properties of this kind of operators we refer the reader to
\cite{BS}.
Now, as usual in this kind of proofs (see e.g. \cite{AM,Ceta}), we proceed by showing the following three claims:
\begin{enumerate}[label={\em(Claim \arabic*)},ref=Claim \arabic*]
{\indentitem
\item \label{it:cl1}  for $k=0$, the equation \eqref{eq:ndeg}  has only the trivial solution in $H^1(\mathbb{R}_+)$;}
{\indentitem
\item \label{it:cl2} for $k=1$, the solutions of \eqref{eq:ndeg}  in
$H^1(\mathbb{R}_+)$ are of the form $c\mathfrak{g}'$, for $c\in\mathbb{R}$;}
{\indentitem
\item \label{it:cl3}  for $k\geq 2$, the equation \eqref{eq:ndeg}  has only the trivial solution in $H^1(\mathbb{R}_+)$.}
\end{enumerate}
{\em Proof of \ref{it:cl1}}. Let $k=0$ and $\psi_0\in H^1(\mathbb{R}_+)$ be a nonzero solution of \eqref{eq:ndeg}. 
The relation $A_0(\mathfrak{g})= - 2 \mathfrak{g}$ and the positivity of $\mathfrak{g}$ imply that the Gausson is the first eigenfunction and then  $\psi_0$ has to change sign.
Thus, by Sturm-Liouville theory $\psi_0$ is unbounded and we get the contradiction.\\
{\em Proof of \ref{it:cl2}}.
First notice that an easy calculation shows that $A_1(\mathfrak{g}')=0$ and 
$\mathfrak{g}'\in H^1(\mathbb{R}_+)$. If we look for a second solution of 
the equation $A_1(\psi)=0$ in the form $\psi(r)=c(r)\mathfrak{g}'(r)$ we have that the function $c$ has to satisfy
\[
r c'' + (n+1-2r^2)c'=0,
\]
and then, in turn,
\[
c(r)=c_1 \Phi(r) + c_2,
\quad
\hbox{ where $\Phi$ is primitive of }
r\mapsto r^{-n-1}e^{r^2}.
\]
Then $c(r)\mathfrak{g}'(r)\to +\infty$ as $r\to +\infty$ if $c_1\neq 0$ and 
this implies that the unique possible choice to get solutions 
of the form $c(r)\mathfrak{g}'(r)$ is to take $c(r)$ constant, proving the claim.\\
{\em Proof of \ref{it:cl3}}.
Since $\lambda_k=\lambda_1+\delta_k$ with $\delta_k>0$ and, by \ref{it:cl2}, the operator $A_1$ is a non-negative operator, then, if $k\geq 2$, $A_k=A_1+\frac{\delta_k}{r^2}$ is a positive operator and so $A_k(\psi)=0$ implies that $\psi=0$. Thus for every $k\geq 2$ and $h\in\{1,\ldots,\mu_k-\mu_{k-2}\}$, we have that $\psi_{k,h}=0$. \\
\noindent
Finally we can conclude observing that, summarizing the previous results, we have that
\[
\operatorname{Ker}(L)
=\operatorname{span}\left\{\mathfrak{g}'Y_{1,h}\right\}
=\operatorname{span}\left\{\partial_{x_h}\mathfrak{g}\right\}.
\]

\subsection{Minimization on $L^2$-spheres}
Let $J$ be as in~\eqref{functional}, $W$ as in \eqref{Wcaz}  and set
\begin{align*}
\mathcal{M}_\nu &:=\big\{ u\in W
\,\vert\, \|u\|^2_2=\nu \big\}, \\
\mathcal{N}_\omega &:=\Big\{ u\in W \setminus \{ 0 \}
\,\vert\, \|\nabla u\|^2_2
+\omega\|u\|_2^2=\int u^2 \log u^2
\Big\}, \\
{E}(u)&:=\frac{1}{2}\|\nabla u\|_2^2-\frac{1}{2}\int u^2 
\log u^2.
\end{align*}
We recall that the functionals $J$ and $E$ are of class $C^1$ in $W$ (see \cite{cazNA}).
We say that a {\em ground state solution} $u$ of 
\eqref{problema} is a solution of the following 
minimization problem
\begin{equation}\label{nehari}
J(u)=m_\omega =\inf_{\mathcal{N}_\omega} J.
\end{equation}
We also set
\begin{equation}
\label{ldue}
c_\nu=\inf_{\mathcal{M}_\nu} E.
\end{equation}
Consider now the sequence $\{u_k\}$ of solutions of \eqref{problema} found in Theorem \ref{main}. Proceeding as in  \cite{JJtanak} we have that the first element $u_1$ of such sequence  solves problem \eqref{nehari}. 
Furthermore, by \cite{BJM},
it follows that $u_1$ has a fixed sign. Then $u_1$ is the Gausson for \eqref{problema} and it belongs to $W$.
We have the following property that can be useful in the study of the orbital stability of ground states for the equation \eqref{sch-evolv}.

\begin{theorem}
\label{biject}
For every $\omega\in{\mathbb R},$ we have
\(
\inf_{\mathcal{N}_\omega} J
=\inf_{\mathcal{M}_{2m_\omega}} J.
\)
\end{theorem}

\noindent
This result follows a particular case of the following
\begin{lemma}
The critical levels of $J$ on $\mathcal{N}_\omega$ are one-to-one with  the critical levels of $E$ on $\mathcal{M}_\nu$.
\end{lemma}
\begin{proof} 
Let $m$ be a critical level of $J$ on $\mathcal{N}_\omega$. We prove 
that $m$ uniquely detects a critical level $c$ of $E$ on $\mathcal{M}_\nu$ and
we have
\begin{equation}
\label{mappatura}
c= \dfrac{\nu}{2}\Big(
\log\dfrac{2m}{\nu} -\omega \Big).
\end{equation}
Let $u$ be a constrained critical point of
$J$ on the manifold $\mathcal{N}_\omega$ such that
$J(u)=m$, then 
$$
\|\nabla u\|_2^2 +(\omega+1)\|u\|_2^2
=\int u^2 \log u^2 + 2 m.
$$
Moreover
$$
\|\nabla u\|^2_2
+\omega\|u\|_2^2=\int u^2 \log u^2
$$
and, by the Pohozaev identity we have 
\begin{equation*}
\frac{n-2}{n}\|\nabla u\|^2_2
+(\omega+1)\|u\|_2^2=
\int u^2 \log u^2.
\end{equation*}
The three identities above give the following
action ripartition
\begin{equation*}
\|u\|^2_2=2m,\qquad
\|\nabla u\|^2_2 = n m,\qquad
\int u^2 \log u^2=(2\omega+n)m.
\end{equation*}
Let us consider $u_{\mu}(x)=\mu u(x)$ with $\mu\in{\mathbb R}^*$.
We notice that $u_\mu$ solves
\[
-\Delta u + (\omega+\log\mu^2) u = u \log u^2.
\]
Moreover $u_{\mu}\in  \mathcal{M}_\nu$ if $2\mu^2 m=\nu$ and then we obtain~\eqref{mappatura}
concluding the proof.
\end{proof}

\begin{proof}[Proof of Theorem~\ref{biject}]
For the minimization problems \eqref{nehari} and
\eqref{ldue}, choosing $\nu=2m_\omega$, it follows
$c_\nu=-\omega \nu/2$,
which yields immediately the last assertion.
\end{proof}

\medskip

\subsection*{Acknowledgments} 
The authors would like to thank Marco Degiovanni and Tobias Weth for providing helpful comments 
about the paper.

\bigskip
\medskip


\begin{thebibliography}{99}


\bibitem{AM}
A. Ambrosetti, A. Malchiodi, {\it Perturbation methods and semilinear elliptic problems on $\mathbb{R}^n$}, Progress in Mathematics {\bf 240}, Birkh\"auser Verlag, Basel, 2006, xii+183 pp. 

\bibitem{BL1}
H. Berestycki, P.L. Lions, {\it Nonlinear scalar field equations.
I. Existence of a ground state}, Arch. Rational Mech. Anal., {\bf
82}, (1983), 313--345.

\bibitem{BL2}
H. Berestycki, P.L. Lions, {\it Nonlinear scalar field equations.
II. Existence of infinitely many solutions}, Arch. Rational Mech.
Anal., {\bf 82}, (1983), 347--375.

\bibitem{BS} 
F.A.\ Berezin, M.A.\ Shubin, The Schr\"odinger equation. Mathematics and its 
Applications (Soviet Series), 66. Kluwer Academic Publishers Group, Dordrecht, 1991, 555pp.


\bibitem{birula76}
I. Bia\l ynicki-Birula, J. Mycielski, {\it Nonlinear wave mechanics}, Ann. Physics {\bf 100} (1976),  62--93.

\bibitem{birula79}
I. Bia\l ynicki-Birula, J. Mycielski,  {\it Gaussons: solitons of the logarithmic Schr\"odinger equation}, Special issue on solitons in physics, Phys. Scripta 20 (1979),  539--544. 

\bibitem{BrezBrow}
H.\ Brezis, F.\ Browder, 
Sur une propri\'et\'e des espaces de Sobolev,
C.R.\ Acad.\ Sci.\ Paris S\'er. A-B {\bf 287} (1978), A113--A115. 

\bibitem{BJ}
J.C. Bronski, R.L. Jerrard, {\it Soliton dynamics in a potential}, Math. Res. Lett. 7 (2000), 329--342.

\bibitem{BJM} 
J. Byeon, L. Jeanjean, M. Mari\c{s}, {\it Symmetry and monotonicity of least energy solutions}, Calc. Var. Partial Differential Equations {\bf 36} (2009), 481--492

\bibitem{CamDeg}  
I. Campa, M. Degiovanni, {\it Subdifferential calculus and nonsmooth critical point theory}, SIAM J. Optim. {\bf 10} (2000), 1020--1048.

\bibitem{cazNA} 
T. Cazenave, {\it Stable solutions of the logarithmic Schr\"odinger equation}, Nonlinear Anal. {\bf 7} (1983), 1127--1140.

\bibitem{caz1} 
T. Cazenave, {\it An introduction to nonlinear
Schr\"odinger equations}, Textos de M\'etodos Matem\'aticos
{\bf 26}, Universidade Federal do Rio de Janeiro 1996.

\bibitem{cazhar} 
T. Cazenave, A. Haraux, {\it \'Equations d'\'evolution avec non lin\'earit\'e logarithmique}, Ann. Fac. Sci. Toulouse Math. (5) {\bf 2} (1980), 21--51. 

\bibitem{cl} 
T. Cazenave, P.L. Lions, {\it Orbital stability 
of standing waves for some nonlinear Schr\"odinger equations}, 
Comm. Math. Phys. {\bf 85} (1982), 549--561.

\bibitem{Ceta}
S.-M. Chang, S. Gustafson, K. Nakanishi, T.-P. Tsai, {\it Spectra of linearized operators for NLS solitary waves}, SIAM J. Math. Anal. {\bf 39} (2008), 1070--1111.

\bibitem{chenli}
W.\ Chen, C.\ Li, {\it Methods on nonlinear elliptic equations}, AIMS Series on Differential Equations \& Dynamical Systems, {\bf 4}, American Institute of Mathematical Sciences (AIMS), Springfield (2010), xii+299 pp.


\bibitem{CDM} 
J.N. Corvellec, M. Degiovanni, M. Marzocchi, {\it Deformation properties for continuous functionals and critical point theory}, Topol. Methods Nonlinear Anal. {\bf 1} (1993), 151--171.

\bibitem{DegMar}  
M. Degiovanni, M. Marzocchi, {\it A critical point theory for nonsmooth functionals}, Ann. Mat. Pura Appl. (4) {\bf 167} (1994), 73--100.

\bibitem{DegZan96}  
M. Degiovanni, S. Zani,  {\it Euler equations involving nonlinearities without growth conditions}, Potential Anal. {\bf 5} (1996), 505--512.

\bibitem{DegZan2000}  
M. Degiovanni, S. Zani, {\it Multiple solutions of semilinear elliptic equations with one-sided growth conditions}, Nonlinear operator theory, Math. Comput. Modelling {\bf 32} (2000), 1377--1393. 

\bibitem{GNN}
B.\ Gidas, W.M.\ Ni, L.\ Nirenberg, 
{\it Symmetry of positive solutions of nonlinear elliptic equations in $\R^n$}, 
Math.\ Anal.\ Appl., Part A, pp. 369--402, 
Adv.\ in Math. Suppl. Stud., 7a, Academic Press, New York, 1981

\bibitem{guerrero} 
P. Guerrero, J.L. L\`opez, J. Nieto,  {\it Global $H^1$ solvability of the 3D logarithmic Schr\"odinger equation}, Nonlinear Anal. Real World Appl. {\bf 11} (2010), 79--87. 

\bibitem{JJtanak}
L.\ Jeanjean, K. Tanaka, 
{\it A remark on least energy solutions in ${\mathbb R}^N$},
Proc. Amer. Math. Soc. {\bf 131} (2003), 2399--2408.

 
\bibitem{LLoss} 
E.H.\ Lieb, M. Loss, {\it Analysis}, Second edition. Graduate Studies in Mathematics 14, American Mathematical Society, Providence, RI, 2001.


\bibitem{serrintang} 
J. Serrin, M. Tang,  {\it Uniqueness of ground states for quasilinear elliptic equations}, Indiana Univ. Math. J. {\bf 49} (2000), 897--923. 

\bibitem{CCM} 
M. Squassina, {\it Radial symmetry of minimax critical points for nonsmooth functionals}, Commun. Contemp. Math. {\bf 13} (2011), 487--508.


\bibitem{struwe}
M.\ Struwe, Variational methods. Applications to nonlinear partial 
differential equations and Hamiltonian systems. 
Springer-Verlag, Berlin, 1990. xiv+244p.

\bibitem{Vazquez}
J.L. V\'azquez,  {\it A strong maximum principle for some quasilinear elliptic equations}, Appl. Math. Optim. {\bf 12} (1984), 191--202.

\bibitem{Zlosh}
K.G. Zloshchastiev, {\it Logarithmic nonlinearity in theories of quantum gravity: origin of time and observational consequences}, Grav. Cosmol. {\bf 16} (2010), 288--297.



\end{thebibliography}
\end{document}